\documentclass{amsart}
\usepackage{amsfonts}
\usepackage{amssymb}
\usepackage{amsmath}

\newtheorem{theorem}{Theorem}[section]

\newtheorem{corollary}{Corollary}[section]

\newtheorem{example}{Example}[section]
\numberwithin{equation}{section}
\def\({\left ( }
\def\){\right )}
\def\<{\left < }
\def\>{\right >}

\begin{document}
\title[Lightlike surfaces]{Lightlike surfaces with planar normal sections in
Minkowski $3-$ space}
\author{Feyza Esra Erdo\u{g}an}
\address{Faculty of Arts and Science, Department of Mathematics, Ad\i yaman
University, 02040 Ad\i yaman, TURKEY}
\email{ferdogan@adiyaman.edu.tr}
\author{Bayram \c{S}ah\.{I}n$^*$}
\address{Department of Mathematics, \.{I}n\"{o}n\"{u} University, 44280\ \
Malatya, TURKEY}
\email{bayram.sahin@inonu.edu.tr}
\author{R\i fat G\"{u}ne\c{s}}
\address{Department of Mathematics, \.{I}n\"{o}n\"{u} University, 44280\ \
Malatya, TURKEY}
\email{rifat.gunes@inonu.edu.tr}

\begin{abstract}
In this paper we study lightlike surfaces of Minkowski $3-$ space such that
they have degenerate or non-degenerate planar normal sections. We first show
that every lightlike surface of Minkowski $3-$ space has degenerate planar
normal sections. Then we study lightlike surfaces with non-degenerate planar
normal sections and obtain a characterization for such lightlike surfaces
\end{abstract}

\maketitle

\section{Introduction}

Surfaces with planar normal sections\textbf{\ }in Euclidean spaces were
first studied by Bang-Yen Chen \cite{Bang-yen Chen}. Later such surfaces or
submanifolds have been studied by many authors \cite{Bang-yen Chen}, \cite%
{Young Ho Kim1}, \cite{Young Ho Kim2}, \cite{Shi-jie Li},\cite{Young Ho Kim3}%
. In \cite{Young Ho Kim2}, Y. H. Kim initiated the study of semi-Riemannian
setting of such surfaces. But as far as we know, lightlike surfaces with
planar normal sections have not been studied so far. Therefore, in this
paper we study lightlike surfaces with planar normal sections of $\mathbb{R}%
^3_1$. \newline

We first define the notion of surfaces with planar normal sections as
follows. Let $M$ be a lightlike surface of $\mathbb{R}^3_1$. For a point $p$
in $M$ and a lightlike vector $\xi $ which spans the radical distribution of
a lightlike surface, the vector $\xi $ and transversal space $tr(TM)$ to $M$
at $p$ determine a 2- dimensional subspace $E(p,\xi )$ in $\mathbb{R}^3_1$
through $p$. The intersection of $M$ and $E(p,\xi )$ gives a lightlike curve
$\gamma $ in a neighborhood of $p,$ which is called the normal section of $M$
at the point $p$ in the direction of $\xi $.\newline

For non-degenerate planar normal sections, we present the following notion.
Let $w$ be a spacelike vector tangent to $M$ at $p$ which spans the chosen
screen distribution of $M$. Then the vector $w$ and transversal space $%
tr(TM) $ to $M$ at $p$ determine a 2- dimensional subspace $E(p,w)$ in $%
\mathbb{R}^3_1$ through $p$. The intersection of $M$ and $E(p,w)$ gives a
spacelike curve $\gamma $ in a neighborhood of $p$ which is called the
normal section of $M$ at $p$ in the direction of $w$. According to both
identifications above, $M$ is said to have degenerate pointwise and
spacelike pointwise planar normal sections, respectively if each normal
section $\gamma $ at $p$ satisfies $\gamma ^{\prime }\wedge \gamma ^{\prime
\prime }\wedge \gamma ^{\prime \prime \prime }=0$ at for each $p$ in $M$.

For a lightlike surface with degenerate planar normal sections, in fact, we
show that every lightlike surface of Minkowski $3-$ space has degenerate
planar normal sections. Then for a lightlike surface with non-degenerate
planar normal sections, we obtain two characterizations.

\noindent \line(1,0){100}\newline
{\footnotesize \textit{\small * Corresponding author}} \newpage

We first show that a lightlike surface $M$ in $\mathbb{R}^3_1$ is a
lightlike surface with non-degenerate planar sections if and only if $M$ is
either screen conformal and totally umbilical or $M$ is totally geodesic. We
also obtain a characterization for non-umbilical screen conformal lightlike
surface with non-degenerate planar normal sections.

\section{Preliminaries}

Let $(\bar{M},\bar{g})$ be an $(m+2)$-dimensional semi-Riemannian manifold
with the indefinite metric $\bar{g}$ of index $q\in \{1,...,m+1\}$ and $M$
be a hypersurface of $\bar{M}$. We denote the tangent space at $x \in M$ by $%
T_x M$. Then
\begin{equation*}
T_x M^{\perp}=\{V_x \in T_x \bar{M}|\bar{g}_x(V_x,W_x)=0, \forall W_x \in
T_x M\}
\end{equation*}
and
\begin{equation*}
RadT_x M=T_x M \cap T_x M^{\perp}.
\end{equation*}
Then, $M$ is called a lightlike hypersurface of $\bar{M}$ if $RadT_x M \neq
\{0\}$ for any $x \in M$. Thus $TM^{\perp}=\bigcap_{x\in M} T_x M^{\perp}$
becomes a one- dimensional distribution $Rad TM$ on $M$. Then there exists a
vector field $\xi \neq 0$ on $M$ such that%
\begin{equation*}
g\left( \xi ,X\right) =0,\newline
\ \ \forall X\in \Gamma \left( TM\right),
\end{equation*}
where $g$ is the induced degenerate metric tensor on $M$. We denote $F(M)$
the algebra of differential functions on $M$ and by $\Gamma(E)$ the $F(M)$-
module of differentiable sections of a vector bundle $E$ over $M$.\newline

A complementary vector bundle $S\left( TM\right) $ of $TM^{\perp }=RadTM$ in
$TM$ i,e.,%
\begin{equation}
TM=RadTM\oplus _{orth}S(TM)  \label{1.1}
\end{equation}%
is called a screen distribution on $M$. It follows from the equation above
that $S(TM)$ is a non-degenerate distribution. Moreover, since we assume
that $M$ is paracompact, there always exists a screen $S(TM)$. Thus, along $M
$ we have the decomposition%
\begin{equation}
T\bar{M}_{|M} = S(TM) \oplus_{orth} S(TM)^{\perp}, \quad S(TM) \cap
S(TM)^{\perp} \neq \{0\},  \label{1.2}
\end{equation}
that is, $S(TM)^{\perp }$ is the orthogonal complement to $S(TM)$ in $T\bar{M%
}\mid _{M}$. Note that $S(TM)^{\perp }$ is also a non-degenerate vector
bundle of rank 2. However, it includes $TM^{\perp }=RadTM$ as its sub-bundle.%
\newline

Let $\left( M,g,S(TM)\right) $ be a lightlike hypersurface of a
semi-Riemannian manifold $\left( \bar{M},\bar{g}\right) $. Then there exists
a unique vector bundle $tr(TM)$ of rank 1 over $M$, such that for any
non-zero section $\xi $ of $TM^{\perp }$ on a coordinate neighborhood $%
U\subset M$, there exists a unique section $N$ of $tr(TM)$ on $U$
satisfying: $TM^{\perp }$ in $S(TM)^{\perp }$ and take $V\in \Gamma \left(
F\mid _{U}\right) ,V\neq 0$. Then $\bar{g}\left( \xi ,V\right) \neq 0$ on $U$%
, otherwise $S(TM)^{\perp }$ would be degenerate at a point of $U$ \cite%
{Krishan L. Duggal and Bayram Sahin}. Define a vector field%
\begin{equation*}
N=\frac{1}{\bar{g}\left( V,\xi \right) }\left\{ V-\frac{\bar{g}\left(
V,V\right) }{2\bar{g}\left( V,\xi \right) }\xi \right\}
\end{equation*}%
on $U$ where $V\in $ $\Gamma \left( F\mid _{U}\right) $ such that $\ \bar{g}%
\left( \xi ,V\right) \neq 0$. Then we have%
\begin{equation}
\bar{g}\left( N,\xi \right) =1,\ \bar{g}\left( N,N\right) =0,\ \bar{g}\left(
N,W\right) =0,\ \forall W\in \Gamma \left( S(TM)\mid _{U}\right)  \label{1.3}
\end{equation}
Moreover, from (\ref{1.1}) and (\ref{1.2}) we have the following
decompositions:\
\begin{equation}
T\bar{M}\mid _{M}=S(TM)\oplus _{orth}\left( TM^{\perp }\oplus tr\left(
TM\right) \right) =TM\oplus tr\left( TM\right)  \label{1.4}
\end{equation}%
Locally, suppose $\left\{ \xi ,N\right\} $ is a pair of sections on $%
U\subset M$ satisfying (\ref{1.3}). Define a symmetric $\digamma \left(
U\right) $-bilinear from $B$ and a 1-form $\tau $ on $U$ . Hence on $U$, for
$X,Y\in \Gamma \left( TM\mid _{U}\right) $%
\begin{eqnarray}
\bar{\nabla}_{X}Y &=&\nabla _{X}Y+B\left( X,Y\right) N  \label{1.5} \\
\bar{\nabla}_{X}N &=&-A_{N}X+\tau \left( X\right) N,  \label{1.6}
\end{eqnarray}%
equations (\ref{1.5}) and (\ref{1.6}) are local Gauss and Weingarten
formulae. Since $\bar{\nabla}$ is a metric connection on $\bar{M},$ it is
easy to see that

\begin{equation}
B\left( X,\xi \right) =0,\forall X\in \Gamma \left( TM\mid _{U}\right) .
\label{1.7}
\end{equation}%
Consequently, the second fundamental form of $M$ is degenerate \cite{Krishan
L. Duggal and Bayram Sahin}. Define a local 1-from $\eta $ by%
\begin{equation}
\eta \left( X\right) =\bar{g}\left( X,N\right) ,\forall \in \Gamma (TM\mid
_{U}).  \label{1.8}
\end{equation}%
Let $P$ denote the projection morphism of $\Gamma \left( TM\right) $ on $%
\Gamma \left( S(TM)\right) $ with respect to the decomposition (\ref{1.1}).
We obtain
\begin{eqnarray}
\nabla_{X}PY &=&\nabla_{X}^{\ast }PY+C\left( X,PY\right) \xi  \label{1.9} \\
\nabla_{X}\xi &=&-A_{\xi }^{\ast }X+\varepsilon \left( X\right) \xi  \notag
\\
&=&-A_{\xi }^{\ast }X-\tau \left( X\right) \xi  \label{1.10}
\end{eqnarray}
where $\nabla _{X}^{\ast }Y$ and $A_{\xi }^{\ast }X$ belong to $\Gamma
\left( S\left( TM\right) \right) ,\nabla $ and $\nabla ^{\ast t}$ are linear
connections on $\Gamma \left( S(TM)\right) $ and $TM^{\perp }$ respectively,
$h^{\ast }$is a $\Gamma \left( TM^{\perp }\right) $-valued $\digamma \left(
M\right) $-bilinear form on $\Gamma \left( TM\right) \times $ $\Gamma \left(
S(TM)\right) $ and $A_{\xi }^{\ast }$ is $\Gamma \left( S(TM)\right) $%
-valued $\digamma \left( M\right) $-linear operator on $\Gamma \left(
TM\right) $. We called them the screen fundamental form and screen shape
operator of $S\left( TM\right) ,$ respectively. Define%
\begin{eqnarray}
C\left( X,PY\right) &=&\bar{g}\left( h^{\ast }(X,PY\right) ,N)\
\label{1.11} \\
\varepsilon \left( X\right) &=&\bar{g}\left( \nabla _{X}^{\ast t}\xi
,N\right) ,\forall X,Y\in \Gamma \left( TM\right) ,\   \label{1.12}
\end{eqnarray}%
one can show that $\varepsilon \left( X\right) =-\tau \left( X\right) $.
Here $C\left( X,PY\right) $ is called the local screen fundamental form of $%
S(TM)$. Precisely, the two local second fundamental forms of $M$ and $S(TM)$
are related to their shape operators by%
\begin{eqnarray}
B\left( X,Y\right) &=&\bar{g}\left( Y,A_{\xi }^{\ast }X\right) ,\
\label{1.13} \\
A_{\xi }^{\ast }\xi &=&0\ ,  \label{1.14} \\
\bar{g}\left( A_{\xi }^{\ast }PY,N\right) &=&0\ ,  \label{1.15} \\
C\left( X,PY\right) &=&\bar{g}\left( PY,A_{N}X\right) ,  \label{1.16} \\
\bar{g}\left( N,A_{N}X\right) &=&0.\   \label{1.17}
\end{eqnarray}%
 A lightlike hypersurface $\left(M,g,S(TM)\right) $ of a semi-Riemannian manifold is called totally umbilical\cite{Krishan L. Duggal and Bayram Sahin}
if there is a smooth function $\varrho ,$ such that%
\begin{equation}
B\left( X,Y\right) =\varrho g\left( X,Y\right) ,\forall X,Y\in \Gamma \left(
TM\right)  \label{1.18}
\end{equation}%
where $\varrho $ is non-vanishing smooth function on a neighborhood $U$ in $M
$.

\qquad A lightlike hypersurface $\left( M,g,S(TM)\right) $ of a
semi-Riemannian manifold is called screen locally conformal if the shape
operators $A_{N}$ and $A_{\xi }^{\ast }$ of $M$ and $S(TM)$, respectively,
are related by%
\begin{equation}
A_{N}=\varphi A_{\xi }^{\ast }  \label{1.19}
\end{equation}%
where $\varphi $ is non-vanishing smooth function on a neighborhood $U$ in $M
$. Therefore, it follows that $\forall X,Y\in \Gamma \left( S\left(
TM\right) \right) ,$ $\xi \in RadTM$%
\begin{equation}
C\left( X,\xi \right) =0,  \label{1.20}
\end{equation}%
For details about screen conformal lightlike hypersurfaces, see: \cite{A-D} and \cite%
{Krishan L. Duggal and Bayram Sahin}  .

\section{Planar normal sections \ of lightlike surfaces in $\mathbb{R}^3_1$}

Let $M$ be a lightlike surface of \textbf{\ }$\mathbb{R}^3_1$. Now we
investigate lightlike surfaces with degenerate planar normal sections. If $%
\gamma $ is a null curve, for a point $p$ in $M,$ we have
\begin{eqnarray}
\gamma ^{\prime }\left( s\right) &=&\xi \   \label{2.1} \\
\gamma ^{\prime \prime }\left( s\right) &=&\bar{\nabla }_{\xi }\xi =-\tau
\left( \xi \right) \xi  \label{2.2} \\
\gamma ^{\prime \prime \prime }\left( s\right) &=&\left[ \xi \left( \tau
\left( \xi \right) \right) +\tau ^{2}\left( \xi \right) \right] \xi
\label{2.3}
\end{eqnarray}%
Then, $\gamma ^{\prime \prime \prime }\left( 0\right) $ is a linear
combination of $\gamma ^{\prime }\left( 0\right) $ and $\gamma ^{\prime
\prime }\left( 0\right) $. Thus (\ref{2.1}), (\ref{2.2}) and (\ref{2.3})
give $\gamma ^{\prime \prime \prime }\left( 0\right) \wedge \gamma ^{\prime
\prime }\left( 0\right) \wedge \gamma ^{\prime }\left( 0\right) =0$. Thus
lightlike surfaces always have planar normal sections.

\begin{corollary}
\textbf{\ }Every lightlike surface of $\mathbb{R}^3_1$ has degenerate planar
normal sections.
\end{corollary}

\bigskip In fact Corollary 3.1 tells us that the above situation is not
interesting. Now, we will check lightlike surfaces with non-degenerate
planar normal sections. Let $M$ be a lightlike hypersurface of \textbf{\ }$%
\mathbb{R}^3_1$. For a point $p$ in $M$ and a spacelike vector $w\in S(TM)$
tangent to $M$ at $p$ , the vector $w$ and transversal space $tr(TM)$ to $M$
at $p$ determine a 2-dimensional subspace $E(p,w)$ in $\mathbb{R}^3_1$
through $p $. The intersection of $M$ and $E(p,w)$ give a spacelike curve $%
\gamma $ in a neighborhood of $p,$ which is called the normal section of $M$
at $p$ in the direction of $w$. Now, we research the conditions for a
lightlike surface of $\mathbb{R}^3_1$ to have non-degenerate planar normal
sections.

\bigskip Let $\left( M,g,S(TM)\right) $ be a totally umbilical and screen
conformal\ lightlike surface of $\left( \bar{g},\mathbb{R}^3_1\right) $. In
this case $S(TM)$ is integrable\cite{A-D}. We denote integral submanifold of $S(TM)$
by $M^{\prime }$. Then, using (\ref{1.6})$,$ (\ref{1.10}) and (\ref{1.19} )
we find
\begin{equation}
C\left( w,w\right) \xi +B\left( w,w\right) N=\bar{g}\left( w,w\right)
\left\{ \alpha \xi +\beta N\right\} =\left\{ \alpha \xi +\beta N\right\} ,
\label{2.4}
\end{equation}%
where $t,$ $\alpha ,\beta \in \mathbb{R}$. In this case, we obtain%
\begin{eqnarray}
\gamma ^{\prime }\left( s\right) &=&w  \label{2.5} \\
\gamma ^{\prime \prime }\left( s\right) &=&\bar{\nabla}_{w}w=\nabla
_{w}^{\ast }w+C\left( w,w\right) \xi +B\left( w,w\right) N  \label{2.6} \\
\gamma ^{\prime \prime }\left( s\right) &=&\nabla _{w}^{\ast }w+\alpha \xi
+\beta N\   \label{2.7} \\
\gamma ^{\prime \prime \prime }\left( s\right) &=&\nabla _{w}^{\ast }\nabla
_{w}^{\ast }w+C\left( w,\nabla _{w}^{\ast }w\right) \xi +w\left( C\left(
w,w\right) \right) \xi -C\left( w,w\right) A_{\xi }^{\ast }w  \label{2.8} \\
&&+w\left( B\left( w,w\right) \right) N-B\left( w,w\right) A_{N}w+B\left(
w,\nabla _{w}^{\ast }w\right) N  \notag \\
\gamma ^{\prime \prime \prime }\left( s\right) &=&\nabla _{w}^{\ast }\nabla
_{w}^{\ast }w+t\left\{ \alpha \xi +\beta N\right\} -\alpha A_{\xi }^{\ast
}w-\beta A_{N}w.  \label{2.9}
\end{eqnarray}%
Where $\nabla ^{\ast }$ and $\nabla $ are linear connections on $S(TM)$ and $%
\Gamma \left( TM\right) $, respectively and $\gamma ^{\prime }\left(
s\right) =w$. From the definition of planar normal section and that $%
S(TM)=Sp\left\{ w\right\} ,$ we have%
\begin{equation}
w\wedge \nabla _{w}^{\ast }w=0  \label{2.10}
\end{equation}%
and%
\begin{equation}
w\wedge \nabla _{w}^{\ast }\nabla _{w}^{\ast }w=0.  \label{2.11}
\end{equation}%
Then, from (\ref{2.4}), (\ref{2.6}), (\ref{2.8}) and (\ref{2.10}), (\ref%
{2.11}) we obtain $\gamma ^{\prime \prime \prime }\left( s\right) \wedge
\gamma ^{\prime \prime }\left( s\right) \wedge \gamma ^{\prime }\left(
s\right) =0$. Thus, $M$ has planar non-degenerate normal sections.\newline

If $M$ is totally geodesic lightlike surface of $\mathbb{R}^3_1$. Then, we
have $B=0$, $A_{\xi }^{\ast }=0$. Hence (\ref{2.5})-(\ref{2.8}) become%
\begin{eqnarray*}
\gamma ^{\prime }\left( s\right) &=&w \\
\gamma ^{\prime \prime }\left( s\right) &=&\nabla _{w}^{\ast }w+\alpha \xi \\
\gamma ^{\prime \prime \prime }\left( s\right) &=&\nabla _{w}^{\ast }\nabla
_{w}^{\ast }w+t\alpha \xi -\beta A_{N}w.
\end{eqnarray*}%
Since $A_{N}w\in \Gamma \left( TM\right) ,$ we have $\gamma ^{\prime \prime
\prime }\left( s\right) \wedge \gamma ^{\prime \prime }\left( s\right)
\wedge \gamma ^{\prime }\left( s\right) =0$.

\bigskip Conversely, we assume that $M$ has planar non-degenerate normal
sections. Then, from (\ref{2.5}), (\ref{2.6}), (\ref{2.8}) and (\ref{2.10}),
(\ref{2.11}) we obtain%
\begin{equation*}
(C\left( w,w\right) \xi +B\left( w,w\right) N)\wedge \left( C\left(
w,w\right) A_{\xi }^{\ast }w+B\left( w,w\right) A_{N}w\right) =0,
\end{equation*}%
thus $\left( C\left( w,w\right) A_{\xi }^{\ast }w+B\left( w,w\right)
A_{N}w\right) =0$ or $C\left( w,w\right) \xi +B\left( w,w\right) N=0$. If%
\newline
$C\left( w,w\right) A_{\xi }^{\ast }w+B\left( w,w\right) A_{N}w=0,$ then,
from
\begin{equation*}
A_{\xi }^{\ast }w=-\frac{B\left( w,w\right) }{C\left( w,w\right) }A_{N}w
\end{equation*}%
at $p\in M,$ $M$ is a screen conformal lightlike surface with $C\left(
w,w\right) \neq 0$. If $C\left( w,w\right) \xi +B\left( w,w\right) N=0$,
then $RadTM$ is parallel and $M$ is totally geodesic.

\bigskip Consequently, we have the following,

\begin{theorem}
Let $M$ be a lightlike surface of $\mathbb{R}^3_1$. Then $M$ has
non-degenerate planar normal sections if and only if \ either $M$ is
umbilical and screen conformal or $M$ is totally geodesic.
\end{theorem}

\begin{theorem}
Let\textbf{\ }$\left( M,g,S(TM)\right) $ be a screen conformal non-umbilical
lightlike surface of $\mathbb{R}^3_1$. Then, for $T\left( w,w\right)
=C\left( w,w\right) \xi +B\left( w,w\right) N$ the following statements are
equivalent

\begin{enumerate}
\item $\left( \bar{\nabla}_{w}T\right) \left( w,w\right) =0$, every
spacelike vector $w\in S(TM)$

\item $\bar{\nabla}T=0$

\item $M$ has non-degenerate planar normal sections and each normal section
at $p$ has one of its vertices at $p$
\end{enumerate}

By the vertex of curve $\gamma \left( s\right) $ we mean a point $p$ on $%
\gamma $ such that its curvature $\kappa $ satisfies $\frac{d\kappa
^{2}\left( p\right) }{ds}=0,$ $\kappa ^{2}=\left\langle \gamma ^{\prime
\prime }\left( s\right) ,\gamma ^{\prime \prime }\left( s\right)
\right\rangle $.
\end{theorem}

\begin{proof}
From (\ref{2.5}), (\ref{2.6}) and that a screen conformal $M$, we have%
\begin{equation*}
\left( \bar{\nabla}_{w}T\right) \left( w,w\right) =\bar{\nabla}_{w}T\left(
w,w\right)
\end{equation*}%
which shows $\left( a\right) \Leftrightarrow \left( b\right) $. $\left(
b\right) $ $\Rightarrow \left( c\right) $ Assume that $\bar{\nabla}T=0$ . If
$\bar{\nabla}T=0$ then $M$ is totally geodesic and Theorem 3.1 implies that $%
M$ has (pointwise) planar normal sections. Let the $\gamma \left( s\right) $
be a normal section of $M$ at $p$ in a given direction $w\in S(TM)$. Then (%
\ref{2.5}) shows that the curvature $\kappa \left( s\right) $ of $\gamma
\left( s\right) $ satisfies%
\begin{eqnarray}
\kappa ^{2}\left( s\right) &=&\left\langle \gamma ^{\prime \prime }\left(
s\right) ,\gamma ^{\prime \prime }\left( s\right) \right\rangle \   \notag \\
&=&2C(w,w)B(w,w)  \notag \\
&=&\left\langle T\left( w,w\right) ,T\left( w,w\right) \right\rangle \
\label{2.12}
\end{eqnarray}%
where $w=\gamma ^{\prime }\left( s\right) $. Therefore we find%
\begin{equation}
\frac{d\kappa ^{2}\left( p\right) }{ds}=\left\langle \bar{\nabla}_{w}T\left(
w,w\right) ,T\left( w,w\right) \right\rangle =\left\langle \left( \bar{\nabla%
}_{w}T\right) \left( w,w\right) ,T\left( w,w\right) \right\rangle \
\label{2.13}
\end{equation}%
Since $\bar{\nabla}_{w}T\left( w,w\right) =0,$ this implies%
\begin{equation*}
\frac{d\kappa ^{2}\left( 0\right) }{ds}=0
\end{equation*}%
at $p=\gamma \left( 0\right) $. Thus $p$ is a vertex of the normal section $%
\gamma \left( s\right) $. $\left( c\right) \Rightarrow \left( a\right) :$ If
$M$ has planar normal sections, then Theorem 3.1 gives%
\begin{equation}
T\left( w,w\right) \wedge \left( \bar{\nabla}_{w}T\right) \left( w,w\right)
=0.  \label{2.14}
\end{equation}%
If $p$ is a vertex of $\gamma \left( s\right) $, then we have%
\begin{equation*}
\frac{d\kappa ^{2}\left( 0\right) }{ds}=0.
\end{equation*}%
Thus, since $M$ has planar normal sections using (\ref{2.13}) we find%
\begin{eqnarray*}
\gamma ^{\prime }\left( s\right) \wedge \gamma ^{\prime \prime }\left(
s\right) \wedge \gamma ^{\prime \prime \prime }\left( s\right) &=&w\wedge
\left( \nabla _{w}^{\ast }w+T\left( w,w\right) \right) \\
&&\wedge \left( \nabla _{w}^{\ast }\nabla _{w}^{\ast }w+tT\left( w,w\right)
+\left( \bar{\nabla}_{w}T\right) \left( w,w\right) \right) =0 \\
\gamma ^{\prime }\left( s\right) \wedge \gamma ^{\prime \prime }\left(
s\right) \wedge \gamma ^{\prime \prime \prime }\left( s\right) &=&T\left(
w,w\right) \wedge \left( \bar{\nabla}_{w}T\right) \left( w,w\right) =0
\end{eqnarray*}%
and
\begin{equation}
\left\langle \left( \bar{\nabla}_{w}T\right) \left( w,w\right) ,T\left(
w,w\right) \right\rangle =0.  \label{2.15}
\end{equation}%
Combining (\ref{2.14}) and (\ref{2.15}) we obtain $\left( \bar{\nabla}%
_{w}T\right) \left( w,w\right) =0$ or $T\left( w,w\right) =0$. Let us define
$U=\left\{ w\in S(TM)\mid T(w,w)=0\right\} $. If $int(U)\neq \varnothing ,$
we obtain $\left( \bar{\nabla}_{w}T\right) \left( w,w\right) =0$ on $int(U)$%
. Thus, by continuity we have $\bar{\nabla}T=0$.
\end{proof}

\begin{example}
Consider the null cone of $\mathbb{R}^3_1$ given by
\begin{equation*}
\wedge =\left\{ \left( x_{1},x_{2},x_{3}\right) \mid
-x_{1}^{2}+x_{2}^{2}+x_{3}^{2}=0,\newline
x_{1},\newline
x_{2},\newline
x_{3}\in IR\newline
\right\} .
\end{equation*}%
The radical bundle of null cone is
\begin{equation*}
\xi =x_{1}\frac{\partial }{\partial x_{1}}+x_{2}\frac{\partial }{\partial
x_{2}}+x_{3}\frac{\partial }{\partial x_{3}}
\end{equation*}%
and screen distribution is spanned by
\begin{equation*}
Z_{1}=x_{2}\frac{\partial }{\partial x_{1}}+x_{3}\frac{\partial }{\partial
x_{2}}
\end{equation*}%
Then the lightlike transversal vector bundle is given by%
\begin{equation*}
Itr(TM)=Span\{N=\frac{1}{2(-x_{1}^{2}+x_{2}^{2})}\left( x_{1}\frac{\partial
}{\partial x_{1}}+x_{2}\frac{\partial }{\partial x_{2}}-x_{3}\frac{\partial
}{\partial x_{3}}\right) \}.
\end{equation*}%
It follows that the corresponding screen distribution $S(TM)$ is spanned by $%
Z_{1}$. Thus%
\begin{eqnarray*}
\nabla _{\xi }\xi &=&x_{1}\frac{\partial }{\partial x_{1}}+x_{2}\frac{%
\partial }{\partial x_{2}}+x_{3}\frac{\partial }{\partial x_{3}} \\
\bar{\nabla}_{\xi }\nabla _{\xi }\xi &=&x_{1}\frac{\partial }{\partial x_{1}}%
+x_{2}\frac{\partial }{\partial x_{2}}+x_{3}\frac{\partial }{\partial x_{3}}.
\end{eqnarray*}%
Then, we obtain%
\begin{equation*}
\gamma ^{\prime \prime \prime }\left( s\right) \wedge \gamma ^{\prime \prime
}\left( s\right) \wedge \gamma ^{\prime }\left( s\right) =0
\end{equation*}%
which shows that null cone has degenerate planar normal sections.
\end{example}

\begin{example}
\textbf{\ }Let $\mathbb{R}^3_1$ be the space $IR^{3}$ endowed with the semi
Euclidean metric%
\begin{equation*}
\bar{g}(x,y)=-x_{0}y_{0}+x_{1}y_{1}+x_{2}y_{2},(x=(x_{0},x_{1},x_{2})).
\end{equation*}%
The lightlike cone $\wedge _{0}^{2}$ is given by the equation $%
-(x_{0})^{2}+(x_{1})^{2}+\left( x_{2}\right) ^{2}=0$, $x\neq 0$. It is known
that $\wedge _{0}^{2}$ is a lightlike surface of $\mathbb{R}^3_1$ and the
radical distribution is spanned by a global vector field%
\begin{equation}
\xi =x_{0}\frac{\partial }{\partial x_{0}}+x_{1}\frac{\partial }{\partial
x_{1}}+x_{2}\frac{\partial }{\partial x_{2}}  \label{2.16}
\end{equation}%
on $\wedge _{0}^{2}$. The unique section $N$ is given by%
\begin{equation}
N=\frac{1}{2(x_{0})^{2}}\left( -x_{0}\frac{\partial }{\partial x_{0}}+x_{1}%
\frac{\partial }{\partial x_{1}}+x_{2}\frac{\partial }{\partial x_{2}}\right)
\label{2.17}
\end{equation}%
and is also defined. As $\xi $ is the position vector field we get%
\begin{equation}
\bar{\nabla}_{X}\xi =\nabla _{X}\xi =X,\text{ \ \ }\forall X\in \Gamma
\left( TM\right) .  \label{2.18}
\end{equation}%
Then, $A_{\xi }^{\ast }X+\tau \left( X\right) \xi +X=0$. As $A_{\xi }^{\ast
} $ is $\Gamma \left( S(TM)\right) $-valued we obtain%
\begin{equation}
A_{\xi }^{\ast }X=-PX,\text{ \ \ }\forall X\in \Gamma \left( TM\right)
\label{2.19}
\end{equation}%
Next, any $X\in \Gamma \left( S(T\wedge _{0}^{2})\right) $ is expressed by $%
X=X_{1}\frac{\partial }{\partial x_{1}}+X_{2}\frac{\partial }{\partial x_{2}}
$ \ where $(X_{1},X_{2})$ satisfy%
\begin{equation}
x_{1}X_{1}+x_{2}X_{2}=0  \label{2.20}
\end{equation}%
and then%
\begin{eqnarray}
\nabla _{\xi }X &=&\bar{\nabla}_{\xi }X=\overset{2}{\sum\limits_{A=0}}%
\sum\limits_{a=1}^{2}x_{A}\frac{\partial X_{a}}{\partial x_{A}}\frac{%
\partial }{\partial x_{a}},  \label{2.21} \\
\bar{g}\left( \nabla _{\xi }X,\xi \right) &=&\overset{2}{\sum\limits_{A=0}}%
\sum\limits_{a=1}^{2}x_{a}x_{A}\frac{\partial X^{a}}{\partial x_{A}}=-\left(
x_{1}X_{1}+x_{2}X_{2}\right) =0  \label{2.22}
\end{eqnarray}%
where (\ref{2.20}) is derived with respect to $x_{0},x_{1},x_{2}$. It is
known that $\wedge _{0}^{2}$ is a screen conformal lightlike surface with
conformal function $\varphi =\frac{1}{2(x_{0})^{2}}$. We also know that $%
A_{N}\xi =0$. By direct compute we find%
\begin{equation*}
A_{N}X=\frac{1}{2(x_{0})^{2}}A_{\xi }^{\ast }X.
\end{equation*}%
Now we evaluate $\gamma ^{\prime },\gamma ^{\prime \prime }$ and $\gamma
^{\prime \prime \prime }$
\begin{eqnarray*}
\gamma ^{\prime } &=&X=\left( 0,-x_{2},x_{1}\right) \\
\gamma ^{\prime \prime } &=&\nabla _{X}X+B\left( X,X\right) N \\
&=&\frac{1}{2}x_{0}\frac{\partial }{\partial x_{0}}-\frac{3}{2}x_{1}\frac{%
\partial }{\partial x_{1}}+x_{2}\frac{\partial }{\partial x_{2}} \\
\gamma ^{\prime \prime \prime } &=&\bar{\nabla}_{X}\nabla _{X}X+X\left(
B\left( X,X\right) \right) N+B\left( X,X\right) \bar{\nabla}_{X}N \\
&=&\nabla _{X}\nabla _{X}X+B\left( X,\nabla _{X}X\right) N+X\left( B\left(
X,X\right) \right) N-B\left( X,X\right) A_{N}X
\end{eqnarray*}%
using $A_{N}X$ in $\gamma ^{\prime \prime \prime }$ we get
\begin{equation*}
\gamma ^{\prime \prime \prime }=\frac{1}{2}x_{2}\frac{\partial }{\partial
x_{1}}-\frac{1}{2}x_{1}\frac{\partial }{\partial x_{2}}.
\end{equation*}%
Therefore $\gamma ^{\prime \prime \prime }$ and $\gamma ^{\prime }$are
linear dependence at $\forall p\in \wedge _{0}^{2}$ and we have
\begin{equation*}
\gamma ^{\prime }\wedge \gamma ^{\prime \prime }\wedge \gamma ^{\prime
\prime \prime }=0.
\end{equation*}%
Namely, $\wedge _{0}^{2}$ has non-degenerate planar normal sections.
\end{example}

\end{document}